\documentclass[11pt]{article}
\usepackage{amsmath,amssymb, latexsym,amsfonts}

\usepackage[thmmarks,amsmath]{ntheorem}

\usepackage{color}

\title{Extreme Value Estimation for\\ Discretely Sampled Continuous Processes}
\author{Holger Drees\footnote{University of Hamburg, Department of Mathematics,
 SPST, Bundesstr.\ 55, 20146 Hamburg, Germany;
email: drees@math.uni-hamburg.de} \and Laurens de Haan$^\ddag$\footnote{Erasmus University Rotterdam, Department of Economics, Burg.\ Oudlaan 50, 3062PA Rotterdam, The Netherlands; email: ldehaan@ese.eur.nl} \and Feridun Turkman\footnote{University of Lisbon, Department of Statistics, Bloco C6 Piso 4, Campo Grande, 1749-016 Lisboa, Portugal; email: kfturkman@fc.ul.pt}}

\newcommand{\R}{\mathbb{R}}
\newcommand{\N}{\mathbb{N}}

\theoremstyle{change} \theoremsymbol{$\Box$}
\theoremseparator{\quad}

\theorembodyfont{\em}

\newtheorem{theorem}{Theorem}[section]

\newtheorem{remark}[theorem]{Remark}

\newtheorem{lemma}[theorem]{Lemma}

\numberwithin{equation}{section}

\theorembodyfont{\rm}
\newtheorem{example}[theorem]{Example}

\newenvironment{proof}{\noindent{\sc Proof.}}{
      \hspace*{\fill} $\Box$ \vspace{2ex} }

\newenvironment{proofof}{\noindent\sc Proof of}{
    \hspace*{\fill} $\Box$ \vspace{2ex} }

\def\eps{\varepsilon}

\def\Min(#1,#2){#1\wedge #2}
\def\Max(#1,#2){#1\vee #2}

\newcommand{\floor}[1]{\lfloor #1 \rfloor}
\newcommand{\inter}[1]{\langle #1 \rangle}

\def\e{{\rm e}}

\def\rueck{\noindent\hangafter=1 \hangindent=1.3em}

\begin{document}

\maketitle

\begin{abstract}
  In environmental applications of extreme value statistics, the underlying stochastic process is often modeled either as a max-stable process in continuous time/space or as a process in the domain of attraction of such a max-stable process. In practice, however, the processes are typically only observed at discrete points and one has to resort to interpolation  to fill in the gaps. We discuss the influence of such an interpolation on estimators of marginal parameters as well as estimators of the exponent measure. In particular, natural conditions on the fineness of the observational scheme are developed which ensure that asymptotically the interpolated estimators behave in the same way as the estimators which use fully observed continuous processes.
\end{abstract}

\footnote{{\noindent\it Keywords and phrases:}  discrete and continuous sampling,  interpolation,  max-stable process\\
{\it AMS 2010 Classification:} Primary 62G32; Secondary 62G05, 62M30.\\
}

\section{Introduction}

In recent years, it has become common in environmetrics to model extreme events by stochastic processes and  random fields. Often max-stable processes are used to describe e.g.\ large amounts of precipitation (see e.g.\ Buhl and Kl\"{u}ppelberg, 2016, and Lehmann et al., 2016), high temperatures (cf.\ Fuentes et al., 2013, and Dombry et al., 2013) or high wind speeds (see Genton et al., 2015, or Oesting et al., 2017). If the observations are not maxima, but exceedances over high thresholds, this approach is not always appropriate.

Instead,  one may merely assume that the underlying process belongs to the domain of attraction of some max-stable process. Let $X^{(i)}=(X^{(i)}_t)_{t\in[0,1]}$, $1\le i\le n$, denote iid random processes with continuous sample paths. (The index set can easily be generalized to arbitrary compact subsets of $\R^d$.) We assume that there exist functions $(a_t(n))_{t\in[0,1]}$, $(b_t(n))_{t\in[0,1]}$, $n\in\N$, such that
\begin{equation} \label{eq:maxconv}
  \Big( \max_{1\le i\le n} \frac{X_t^{(i)}-b_t(n)}{a_t(n)}\Big)_{t\in[0,1]} \;\longrightarrow\; (Y_t)_{t\in[0,1]}=:Y
\end{equation}
weakly in $C[0,1]$ for some max-stable process $Y$ with non-degenerate margins. In particular, $Y_t$ has an extreme value distribution for each $t\in[0,1]$, and w.l.o.g.\ one may assume that
$$ P\{Y_t\le y\} = \exp\big(-(1+\gamma_t y)^{-1/\gamma_t}\big)=G_{\gamma_t}(y) $$
for all $y$ satisfying $1+\gamma_ty>0$ and some continuous function $(\gamma_t)_{t\in[0,1]}$. Hence, the cdf $F_t$ of $X^{(1)}_t$ belongs to the max domain of attraction of $G_{\gamma_t}$ and one may choose $b_t(n)=U_t(n):= F_t^\leftarrow(1-n^{-1})$, with $F^\leftarrow$ denoting the generalized inverse of a cdf $F$. Indeed, convergence \eqref{eq:maxconv} is equivalent to $F_t\in D(G_{\gamma_t})$ for all $t\in[0,1]$ and the following marginally standardized version
\begin{equation}  \label{eq:normconv}
 \Big( \frac 1n \max_{1\le i\le n} \xi_t^{(i)}\Big)_{t\in[0,1]} \;\longrightarrow\; \big((1+\gamma_t Y_t)^{1/\gamma_t}\big)_{t\in[0,1]}=: (Z_t)_{t\in[0,1]}=:Z
\end{equation}
weakly in $C[0,1]$, where
\begin{equation}  \label{eq:xidef}
 \xi_t^{(i)}:= \frac 1{1-F_t(X^{(i)}_t)}, \quad t\in[0,1], i\in\N
\end{equation}
(de Haan and Ferreira, 2006, Theorem 9.2.1).

The distribution of $Z$ (and thus the dependence structure of $Y$) is determined by the so-called exponent measure $\nu$ via the relation
\begin{equation}  \label{eq:nudef}
   P\{Z\in A\}=\exp(-\nu(A^c))
\end{equation}
for all Borel sets $A\subset C[0,1]$ of the type $A=\big\{f\in C[0,1]\mid f(t)\le x_j,\forall\,t\in K_j, 1\le j\le m\big\}$ for some $m\in\N$, compact sets $K_j\subset [0,1]$ and $x_j\in (0,\infty)$, $1\le j\le m$, provided $\inf\{\|f\|_\infty\mid f\in A^c\}>0$. The extreme value behavior of the process $X^{(1)}$ is thus described by the functions $(\gamma_t)_{t\in[0,1]}$, $(a_t(\cdot))_{t\in[0,1]}$ and $(U_t(\cdot))_{t\in[0,1]}$, and the exponent measure $\nu$. Estimators of these quantities have been proposed  by de Haan and Lin (2003), who also established their consistency. Einmahl and Lin (2006) proved the asymptotic normality of the marginal estimators under suitable conditions; see de Haan and Ferreira (2006), Chapter 10 for details.

All these estimators require that the processes $X^{(i)}$ are observed everywhere. In practice, however, measurements $X_t^{(i)}$, $1\le i\le n$, are often only made at certain discrete points $t_{n,j}$,  $1\le j\le j_n$, e.g.\ where weather stations are located.
If one assumes a fully parametric model for the extreme value behavior of the processes, one may infer the parameters from the discretely sampled observations (for instance, using a composite likelihood approach as in Buhl and Kl\"{u}ppelberg, 2016, or a generalized method of moments like de Haan and Pereira, 2006, or Oesting et al., 2017, in the context of max-stable models) and thus obtain estimators for the extreme value behavior at any point $t$.  If one refrains from making such restrictive assumptions, then one has to rely on interpolation to infer the extreme value behavior of the process at points outside the measurement grid. For particular classes of max-stable processes, such statistical interpolation has been discussed e.g.\ by Falk et al.\ (2015).
 In contrast, in Section 2 we give conditions under which consistency and asymptotic normality of generic estimators of the marginal functions and the exponent measure carry over to discretized versions of these estimators which only use observations $X^{(i)}_{t_{n,j}}$, $1\le i\le n$, $1\le j\le j_n$, in the general setting. In particular, we show that a simple interpolation method works under a stochastic smoothness condition for large values of the process. All proofs are deferred to Section 3.

Interpolation of max-stable process and of more general processes in extreme regions  has also been discussed in different contexts. For instance, Piterbarg (2004) examined when the maxima of a stationary  Gaussian process $Z=Z(t)_{t\in[0,T]}$ on the whole interval $[0,T]$ resp.\ on a discrete grid show the same asymptotic behavior, while Turkman (2012) considered the same problem for more general stationary processes. See also Albin (1990) for results in this spirit. Wang and Stoev (2011), Dombry et al.\ (2013) and Oesting and Schlather (2014), among others, developed algorithms to simulate a max-stable process $Y$ given its values on a finite grid. However, none of these papers dealt with fitting a model for discretely observed processes.

\section{Interpolation estimators}

\subsection{Estimating marginal parameters}

Several estimators of the marginal tail behavior have been discussed in the literature. We focus on estimators which use $k_n+1$ largest order statistics of $X_t^{(i)}$, $1\le i\le n$, denoted by $X_t^{(n-k_n:n)}\le X_t^{(n-k_n+1:n)}\le \cdots\le X_t^{(n:n)}$. Here $(k_n)_{n\in\N}$ is some intermediate sequence, i.e.\ $k_n\in\{1,\ldots,n\}$, $k_n\to\infty$ and $k_n/n\to 0$ as $n\to\infty$. These estimators are motivated by the assumption that, above the quantile $U_t(n/k_n)$, the tail of $F_t$ is well approximated by a generalized Pareto distribution (GPD), that is
$$ 1-F_t(x)\approx \frac {k_n}n \bigg( 1+\gamma_t \frac{x-U_t(n/k_n)}{a_t(n/k_n)}\bigg)^{-1/\gamma_t}, \quad x\ge U_t(n/k_n),
$$
and
$$ U_t(y)\approx U_t(n/k_n)+a_t(n/k_n)\frac{(k_ny/n)^{\gamma_t}-1}{\gamma_t}, \quad y\ge n/k_n. $$
To employ these approximations, for example for statistical inference on extreme quantiles, one needs estimators of $\gamma_t$, $a_t(n/k_n)$ and $U_t(n/k_n)$, $t\in[0,1]$.

De Haan and Lin (2003) and Einmahl and Lin (2006) proved consistency and asymptotic normality, respectively, uniformly for $t\in[0,1]$ for the following set of estimators:
\begin{eqnarray}
\hat\gamma_{n,t} &:= & \hat\gamma_{n,t}^{+}+\hat\gamma_{n,t}^{-}, \label{eq:gammadef}\\
\hat a_{n,t}(n/k_n) & := & X_t^{(n-k_n:n)}\hat\gamma_{n,t}^{+}(1-\hat\gamma_{n,t}^{-}), \label{eq:adef}\\
\hat U_{n,t}(n/k_n) & := & X_t^{(n-k_n:n)}, \label{eq:Udef}
\end{eqnarray}
where
\begin{eqnarray*}
  M_{n,t}^{(j)} & := & \frac{1}{k_n}\sum_{i=1}^{k_n}\Big(\log \frac{X_t^{(n-i+1:n)}}{X_t^{(n-k_n:n)}}\Big)^j,\quad j=1,2,\\
  \hat\gamma_{n,t}^{+} & := & M_{n,t}^{(1)},\\
  \hat\gamma_{n,t}^{-} & := & 1-\frac 12\left(1-\frac{(M_{n,t}^{(1)})^2}{M_{n,t}^{(2)}}\right)^{-1}.
\end{eqnarray*}
Because consistency and asymptotic normality of other estimators can be proved for more general estimators using similar techniques,  here we consider generic estimators $\hat\gamma_{n,t}$, $\hat U_{n,t}(n/k_n)$ and $\hat a_{n,t}(n/k_n)$ that only depend on $X_t^{(i)}$, $1\le i\le n$, for each $t\in [0,1]$, and satisfy the following condition for some positive bounded sequence $(\lambda_n)_{n\in\N}$.\\[1ex]
{\bf (E($\boldsymbol{\lambda}_{\mathbf{n}}$))}\;\; \;\parbox[t]{14.6cm}{There exists versions of the estimators (denoted by the same symbols) and processes $\Gamma$, $A$ and $B$ with continuous sample paths such that
\begin{eqnarray}
  \sup_{t\in[0,1]} \big|\lambda_n^{-1}(\hat\gamma_{n,t}-\gamma_t)-\Gamma_t \big| &\stackrel{(P)}{\longrightarrow} & 0 \label{eq:gammaconv}\\
  \sup_{t\in[0,1]} \bigg|\lambda_n^{-1}\Big(\frac{\hat a_{n,t}(n/k_n)}{a_t(n/k_n)}-1\Big)-A_t \bigg| &\stackrel{(P)}{\longrightarrow}& 0 \label{eq:aconv}\\
  \sup_{t\in[0,1]} \bigg|\lambda_n^{-1}\frac{\hat U_{n,t}(n/k_n)-U_t(n/k_n)}{a_t(n/k_n)}-B_t \bigg| &\stackrel{(P)}{\longrightarrow}& 0 \label{eq:Uconv}
\end{eqnarray}
}\\[1ex]
Note that \eqref{eq:gammaconv}--\eqref{eq:Uconv} imply the joint convergence of the standardized estimation errors for all three processes.
We are mainly interested in two cases. Condition (E(1)) (i.e.\  $\lambda_n=1$ for all $n\in\N$) with $\Gamma\equiv A\equiv B\equiv 0$  means consistency of the estimators, whereas (E($k_n^{-1/2}$)) with a  Gaussian process $(\Gamma,A,B)^T$ states the uniform joint asymptotic normality of the marginal estimators with the usual rate of convergence.

If the processes $X^{(i)}$, $1\le i\le n$, are only observed at points $t_{n,1}<t_{n,2} <\cdots<t_{n,j_n}$ in $[0,1]$, then one has to interpolate the resulting estimators $\hat\gamma_{n,t_{n,j}}$, $\hat a_{n,t_{n,j}}(n/k_n)$ and $\hat U_{n,t_{n,j}}(n/k_n)$ to obtain estimators of the marginal parameters at points $t\in[0,1]\setminus\{t_{n,j}|1\le j\le j_n\}$. The simplest approach is to use the estimator at the closest point of observation, but this results in estimators which (in contrast to the functions to be estimated) are not continuous. Therefore, here we consider linearly interpolated estimators.
For any function $z=(z_t)_{t\in[0,1]}$ and $t\in[0,1]$  let
\begin{equation} \label{eq:interdef}
 \langle z \rangle_{n,t} := \left\{
   \begin{array}{lcl}
     z_{t_{n,1}} & & t\le t_{n,1},\\
     \frac{t_{n,j}-t}{t_{n,j}-t_{n,j-1}} z_{t_{n,j-1}} + \frac{t-t_{n,j-1}}{t_{n,j}-t_{n,j-1}} z_{t_{n,j}} & \text{ if } & t_{n,j-1}<t\le t_{n,j} \text{ for some } 2\le j\le j_n, \\
     z_{t_{n,j_n}} & & t>t_{n,j_n}.
   \end{array}
   \right.
\end{equation}
Then we define  estimators
 \begin{eqnarray*}
   \hat\gamma_{n,t}^* & := & \inter{\hat\gamma_n}_{n,t},\\
   \hat a_{n,t}^*(n/k_n) & := & \inter{\hat a_n(n/k_n)}_{n,t},\\
   \hat U_{n,t}^*(n/k_n) & := & \inter{\hat U_n(n/k_n)}_{n,t}.
 \end{eqnarray*}

We show in Theorem \ref{theo:simple} that asymptotically these ``interpolation estimators'' behave in the same way as the original ones if the functions to be estimated are smooth and the points of observations are sufficiently dense.

Throughout the remainder of the paper, we assume that
$$ \delta_n := \max_{1\le j\le j_n+1} (t_{n,j}-t_{n,j-1})\to 0 $$
as $n$ tends to 0 with $t_{n,0}:=0$ and $t_{n,j_n+1}:=1$. Moreover, we use the notation $\sup_{|s-t|\le\delta_n}$ as a shorthand for $\sup_{s,t\in[0,1], |s-t|\le\delta_n}$.

\begin{theorem} \label{theo:simple}
  If condition (E($\lambda_n$)) holds and
  \begin{eqnarray}
    \sup_{|s-t|\le\delta_n}|\gamma_s-\gamma_t| & = & o(\lambda_n)  \label{eq:gammasmooth}\\
    \sup_{|s-t|\le\delta_n}\Big|\frac{a_s(n/k_n)}{a_t(n/k_n)}-1\Big| & = & o(\lambda_n)  \label{eq:asmooth}\\
    \sup_{|s-t|\le\delta_n}\Big|\frac{U_s(n/k_n)-U_t(n/k_n)}{a_t(n/k_n)}\Big| & = & o(\lambda_n)  \label{eq:Usmooth}
  \end{eqnarray}
  then
\begin{eqnarray}
  \sup_{t\in[0,1]} \big|\lambda_n^{-1}(\hat\gamma_{n,t}^*-\gamma_t)-\Gamma_t \big| \,\stackrel{(P)}{\longrightarrow}\, 0 \label{eq:dgammaconv}\\
  \sup_{t\in[0,1]} \bigg|\lambda_n^{-1}\Big(\frac{\hat a_{n,t}^*(n/k_n)}{a_t(n/k_n)}-1\Big)-A_t \bigg| \,\stackrel{(P)}{\longrightarrow}\, 0 \label{eq:daconv}\\
  \sup_{t\in[0,1]} \bigg|\lambda_n^{-1}\frac{\hat U_{n,t}^*(n/k_n)-U_t(n/k_n)}{a_t(n/k_n)}-B_t \bigg| \,\stackrel{(P)}{\longrightarrow}\, 0. \label{eq:dUconv}
\end{eqnarray}
\end{theorem}
It is easily seen that this result carries over to more refined interpolation schemes, e.g.\ using splines. A close inspection of its proof reveals that one can also generalize the result to multivariate index sets if the following two conditions are fulfilled. First,
an estimator at an arbitrary point $t$ should be a weighted average (with bounded weights) of the corresponding estimators at grid points in a certain neighborhood of $t$. Second, similarly as in \eqref{eq:gammasmooth}--\eqref{eq:Usmooth}, the local fluctuations of the functions $\gamma_\cdot$, $a_\cdot(n/k_n)$ and $U_\cdot(n/k_n)$ over the neighborhoods used in the interpolation scheme must be of smaller order than $\lambda_n$.

Note that, for $\lambda_n\equiv 1$, condition \eqref{eq:gammasmooth} is automatically fulfilled by the continuity of $(\gamma_t)_{t\in[0,1]}$. In contrast, \eqref{eq:asmooth} and \eqref{eq:Usmooth} need not be fulfilled and then the assertions need not hold, as the following example shows.

\begin{example}  \label{ex:counter}
  Let $V_i$, $1\le i\le n$, be iid standard Pareto random variables, i.e.\ $P\{V_i>x\}=x^{-1}$ for all $x\ge 1$, and define $X_t^{(i)}:=V_i^{\gamma_t}$ for some continuous function $t\mapsto\gamma_t>0$. Then obviously $F_t(x)=1-x^{-1/\gamma_t}$, $x\ge 1$, belongs to the domain of attraction of $G_{\gamma_t}$, and one may choose $a_t(y)=\gamma_tU_t(y)=\gamma_ty^{\gamma_t}$ for all $y>0$. Moreover, $\xi_t^{(i)} =1/(1-F_t(X^{(i)}))=V_i$ for all $t\in[0,1]$, and thus \eqref{eq:normconv} trivially holds with $Z_t=Z_0$ for a unit Fr\'{e}chet random variable $Z_0$.

  Now
  $$ \frac{U_s(n/k_n)-U_t(n/k_n)}{a_t(n/k_n)} = \frac 1{\gamma_t} \Big((n/k_n)^{\gamma_s-\gamma_t}-1\Big) $$
  tends to 0 uniformly if and only if
  \begin{equation} \label{eq:gammarate}
    \sup_{|s-t|\le\delta_n} |\gamma_s-\gamma_t| = o\big(1/\log(n/k_n)\big).
  \end{equation}
   It is easily seen that also \eqref{eq:asmooth} is equivalent to \eqref{eq:gammarate}.

   Now to check consistency (in the sense of \eqref{eq:dUconv} with $\lambda_n\equiv 1$ and $B\equiv 0$) of the estimator $\hat U_{n,t}^*(n/k_n)=X^{(n-k_n:n)}_{<t>_n}$ (see \eqref{eq:Udef})  note that for $t_{n,j-1}<t\le t_{n,j}$
  \begin{eqnarray}
     \frac{\hat U_{n,t}^*(n/k_n)-U_t(n/k_n)}{a_t(n/k_n)}
      & = & \frac 1{\gamma_t} \bigg[ \frac{t_{n,j}-t}{t_{n,j}-t_{n,j-1}}\bigg(\Big(\frac{k_n}n V_{n-k_n:n}\Big)^{\gamma_{t_{n,j-1}}} \Big(\frac{k_n}n\Big)^{\gamma_t-\gamma_{t_{n,j-1}}}-1\bigg)  \nonumber\\
      & & \hspace{0.5cm} +
       \frac{t-t_{n,j-1}}{t_{n,j}-t_{n,j-1}}\bigg(\Big(\frac{k_n}n V_{n-k_n:n}\Big)^{\gamma_{t_{n,j}}} \Big(\frac{k_n}n\Big)^{\gamma_t-\gamma_{t_{n,j}}}-1\bigg)\bigg].\hspace*{1cm} \label{eq:counterex}
  \end{eqnarray}

  Because $(k_n/n)V_{n-k_n:n}\to 1$ in probability, $\hat U_{n,t}^*(n/k_n)$ is consistent if
  $\max\big(|\gamma_t-\gamma_{t_{n,j}}|,|\gamma_t-\gamma_{t_{n,j-1}}|\big)=o(1/\log(n/k_n))$. In contrast, if e.g.\ for $t=t_{n,j-1}+c(t_{n,j}-t_{n,j-1})$ (for some $c\in (0,1)$)
  $\log(n/k_n)(\gamma_t-\gamma_{t_{n,j-1}})\to -\infty$, then the right-hand side of \eqref{eq:counterex} tends to $\infty$. In particular, in this case $\hat U_{n,\cdot}^*(n/k_n)$ is not uniformly consistent.

  So, roughly speaking, one needs \eqref{eq:gammasmooth} to hold with $\lambda_n=1/\log(n/k_n)$ to ensure \eqref{eq:dUconv} with $\lambda_n\equiv 1$.
\end{example}

\begin{example}  \label{ex:expgauss}
  Secondly, we consider a generalization of an example examined by Einmahl and Lin (2006). Let $Z=(Z_t)_{t\in[0,1]}$ be a centered Gaussian process such that
  \begin{equation} \label{eq:vardiff}
    E\big((Z_s-Z_t)^2\big)\le C_1|s-t|^{\alpha_1}, \quad \forall s,t\in[0,1],
  \end{equation}
  for some constants $C_1>0$ and $\alpha_1>0$. Moreover, let $t\mapsto \gamma_t$ be a positive function such that $|\gamma_s-\gamma_t|\le C_2|s-t|^{\alpha_2}$ for some $C_2,\alpha_2>0$. For a standard Pareto random variable $Y$ (i.e.\ $P\{Y>x\}=x^{-1}$ for $x>1$) independent of $Z$ define $X_t:= Y^{\gamma_t}\e^{Z_t}$, $t\in [0,1]$.

  It is well known that, under the above conditions, $Z$ has continuous sample paths with \linebreak  $P\{\sup_{t\in[0,1]} Z_t/\gamma_t>x\}\le \exp(-cx^2)$ for some $c>0$ and sufficiently large $x$ (see, e.g., Adler, 1990, Theorem 1.4 and (2.4)). In particular, $E\big(\sup_{t\in[0,1]}\e^{Z_t/\gamma_t}\big)<\infty$. Hence, the example investigated by Einmahl and Lin (2006), pp.\ 477 f., shows that for iid copies $(Y^{(i)},Z^{(i)})$ of $(Y,Z)$
  $$  \bigg(\max_{1\le i\le n}\frac{ Y^{(i)}\exp(Z_t^{(i)}/\gamma_t)}{E\big(\exp(Z_t^{(i)}/\gamma_t)\big)n}\bigg)_{t\in[0,1]} \to \eta
  $$
  for some simple max-stable limit process $\eta$ (i.e., with unit Fr\'{e}chet marginals). Now, the continuous mapping theorem yields convergence (1.1) towards $(\eta_t^{\gamma_t})_{t\in[0,1]}$.

  Let $\sigma_t^2:= Var(Z_t)$. Straightforward calculations show that,  for all $M>0$,
  \begin{eqnarray*}
    P\{X_t>u\} & = & \int P\big\{ Y>u^{1/\gamma_t}\e^{-z/\gamma_t} \big\}\, P^{Z_t}(dz)\\
    & = & u^{-1/\gamma_t} \exp\big(\sigma_t^2/(2\gamma_t^2)\big)\Phi\Big(\frac{\log u}{\sigma_t}-\frac{\sigma_t}{\gamma_t}\Big) + 1-\Phi\Big(\frac{\log u}{\sigma_t}\Big) \\
    & = & u^{-1/\gamma_t} \exp\big(\sigma_t^2/(2\gamma_t^2)\big) + o(u^{-M})
  \end{eqnarray*}
  uniformly for all $t\in[0,1]$ as $u\to\infty$. It can easily be concluded that, for all $\kappa>0$, one has for sufficiently large $x$
  $$ \sup_{t\in[0,1]}  \big|U_t(x)-c_t x^{\gamma_t}\big|\le x^{-\kappa}
  $$
  with $c_t:=\exp(\sigma_t^2/(2\gamma_t))$.

  Since $\gamma_t$ is assumed positive, we can thus choose $a_t(n/k_n)=\gamma_tc_t(n/k_n)^{\gamma_t}$.
  Therefore, it can be shown in a similar way as in Einmahl and Lin (2006) that the estimators \eqref{eq:gammadef}--\eqref{eq:Udef} satisfy condition $E(k_n^{-1/2})$ provided $k_n=o(n^{1-\eps})$ for some $\eps>0$.

  Note that
  $|\sigma_t^2-\sigma_s^2|=\big|E\big((Z_t-Z_s)(Z_t+Z_s)\big)|\le C_3 |t-s|^{\alpha_1/2} $
  for some $C_3>0$ by \eqref{eq:vardiff} and the Cauchy-Schwarz inequality. Hence also $t\mapsto c_t$ is H\"{o}lder continuous with exponent $\alpha:=\min(\alpha_1/2,\alpha_2)$.

  Next we derive a condition on $\delta_n$ which ensures that \eqref{eq:gammasmooth}--\eqref{eq:Usmooth} hold with $\lambda_n=k_n^{-1/2}$.
  Check that for an arbitrarily large $\kappa>0$ one has eventually
  $$ \sup_{|s-t|\le\delta_n}\Big| \frac{U_s(n/k_n)-U_t(n/k_n)}{a_t(n/k_n)}\Big| \le
    \frac{\sup_{|s-t|\le\delta_n} c_s\big|(n/k_n)^{\gamma_s-\gamma_t}-1\big| + |c_s-c_t|
     + 2 (n/k_n)^{-\kappa}}{\inf_{t\in[0,1]} c_t\gamma_t}.
  $$
  Thus, by the H\"{o}lder condition on $\gamma_\cdot$,
  the first term in the numerator is of smaller order than $k_n^{-1/2}$ if
  $\delta_n^{\alpha_2}\log(n/k_n)=o(k_n^{-1/2})$ or, equivalently,
  $
    \delta_n=o\big(k_n^{-1/(2\alpha_2)}(\log n)^{-1/\alpha_2}\big).
  $
  By the H\"{o}lder continuity of $c_\cdot$, the second term is negligible if $\delta_n=o(k_n^{-1/(2\alpha)})$.
  Since $n^\eps=o(n/k_n)$ and $\kappa$ can be chosen larger than $1/\eps$, condition \eqref{eq:Usmooth} thus holds if
  \begin{equation} \label{eq:deltancond}
    \delta_n=o\big(\min\big(k_n^{-1/(2\alpha_2)}(\log n)^{-1/\alpha_2},k_n^{-1/\alpha_1}\big)\big).
  \end{equation}
  Condition \eqref{eq:asmooth} reads as
  $$\sup_{|s-t|\le\delta_n}\Big|(n/k_n)^{\gamma_s-\gamma_t}\frac{c_s\gamma_s}{c_t\gamma_t}-1\Big| = o(k_n^{-1/2}). $$
  Again by the H\"{o}lder continuity of $\gamma_\cdot$ and $c_\cdot\gamma_\cdot$ with exponents $\alpha_2$ and  $\alpha$, respectively, under condition \eqref{eq:deltancond} one has $(n/k_n)^{\gamma_s-\gamma_t}=1+o(k_n^{-1/2})$ and $c_s\gamma_s/(c_t\gamma_t)=1+o(k_n^{-1/2})$ uniformly for $|s-t|\le\delta_n$, and thus \eqref{eq:asmooth} holds.
  Finally, in view of the H\"{o}lder continuity of $\gamma_\cdot$, \eqref{eq:deltancond} obviously also implies \eqref{eq:gammasmooth}.

  Therefore, one may conclude that for sampling schemes such that \eqref{eq:deltancond} is fulfilled the interpolated marginal estimators asymptotically behave in the same way as the estimators considered by Einmahl and Lin (2006).
\end{example}

Theorem \ref{theo:simple} gives sufficient conditions in terms of the smoothness of the marginal functions $\gamma_\cdot$, $a_\cdot(n/k_n)$ and $U_\cdot(n/k_n)$ which ensure that the asymptotic behavior of the marginal estimators carry over to their discretized versions.
In what follows, we replace these purely analytical conditions with two different assumptions which may sometimes be easier to interpret. The first condition quantifies the accuracy of the GPD approximation to the marginal tails, while the second is a smoothness condition on the sample paths in extreme regions. \\[1ex]
{\bf (M($\boldsymbol{\lambda}_{\mathbf{n}}$))}\; \; \; \parbox[t]{14.1cm}{For all $0<y_0<y_1<\infty$
$$ \sup_{t\in[0,1]} \sup_{y\in[y_0,y_1]} \Big| \frac{U_t(yn/k_n)-U_t(n/k_n)}{a_t(n/k_n)} - \frac{y^{\gamma_t}-1}{\gamma_t}\Big| = o(\lambda_n).
$$}\\[1ex]
In the case $\lambda_n\equiv 1$, condition $M(1)$ follows from \eqref{eq:maxconv} and is thus automatically fulfilled in our setting (see de Haan and Ferreira, 2006, Section 9.2).

In what follows, $X$ denotes a process with the same distribution as $X^{(1)}$.\\[1ex]
{\bf (S($\boldsymbol{\lambda}_{\mathbf{n}}$))}\; \; \; \parbox[t]{14.4cm}{There exists a constant $\tau<\tau_{\max}:=\inf_{t\in[0,1]}1/\gamma_t^-$ (i.e.\, $\tau_{\max}=\infty$ if $\gamma\ge 0$ and $\tau_{\max}=1/|\gamma|$ else) such that for all $\eps>0$
$$ \sup_{|s-t|\le \delta_n} P\Big\{ \frac{|X_s-X_t|}{a_t(n/k_n)}>\eps\lambda_n, X_t>U_t(n/k_n)+\tau a_t(n/k_n)\Big\} = o(\lambda_n k_n/n).
$$}\\[1ex]
Note that $P\{X_t>U_t(n/k_n)+\tau a_t(n/k_n)\}\sim (k_n/n) (1+\gamma_t\tau)^{-1/\gamma_t}$. Hence, condition (S(1)) states that the fluctuations of the process in a neighborhood (of the size of the maximal grid width) of some point $t$ where the process is large are of smaller order than the random variability (represented by the scale function $a_t$) at this point. If $\lambda_n$ tends to 0, condition (S($\lambda_n$)) restricts the fluctuations further.
\begin{theorem} \label{theo:mainmarg}
  Assume that $n/k_n=h(n)$ for some function $h$ which is regularly varying with an index $\kappa\in (0,1]$, and that $k_{n-1}/k_n-1=o(\lambda_n)$ and $\sup_{c\in[c_0,1]} \lambda_{\floor{cn}}/\lambda_n=O(1)$ for all $c_0>0$.
  Then, under the
   conditions (M($\lambda_n$)) and (S($\lambda_n$)),
  \begin{equation} \label{eq:Uapprox}
    \sup_{|s-t|\le\delta_n} \sup_{y\in[\underline y,\bar y]} \Big| \frac{U_s(yn/k_n)-U_t(n/k_n)}{a_t(n/k_n)} - \frac{y^{\gamma_t}-1}{\gamma_t}\Big| = o(\lambda_n)
  \end{equation}
  for all $0<\underline y<\bar y<\infty$, and \eqref{eq:gammasmooth}--\eqref{eq:Usmooth} hold.
\end{theorem}
If the condition on the regular variation of $n/k_n$ is fulfilled and $\lambda_n=1$, then the second condition on $k_n$ is automatically fulfilled. Likewise, the condition on $\lambda_n$ follows from the regular variation if $\lambda_n=k_n^{-1/2}$, and it is trivial if $\lambda_n\equiv 1$.

\begin{remark} \rm
  Usually, estimators of the functions $\gamma_\cdot, a_\cdot(n/k_n)$ and $U_\cdot (n/k_n)$ are not of interest of their own, but they are instrumental in estimating parameters with an operational meaning, like extreme quantiles. For example, assume that one wants to determine the threshold at point $t$ which is exceeded with a very small probability $p_n=o(k_n/n)$, that is, we want to estimate $U_t(1/p_n)$. (In environmetrics, such an exceedance is often interpreted as a $1/(mp_n)$-year event if $m$ observations $X^{(i)}$ are made each year.)

  If the full processes $X^{(i)}$ are observed, then a popular estimator is
  $$ \hat x_{n,t} := \hat U_{n,t}(n/k_n) + \hat a_{n,t}(n/k_n)\frac{(np_n/k_n)^{-\hat\gamma_{n,t}}-1}{\hat\gamma_{n,t}}, \quad t\in[0,1].
  $$
  Under condition $E(k_n^{-1/2})$ with a Gaussian limiting process $(\Gamma,A,B)$ and the additional condition
  $$ \sup_{t\in[0,1]}\Big|\frac{U_t(1/p_n)-U_t(n/k_n)}{a_t(n/k_n)}
     -\frac{(np_n/k_n)^{-\gamma_t}-1}{\gamma_t} \Big| = o(k_n^{-1/2})
  $$
  the uniform asymptotic normality of $\hat x_{n,t}$, $t\in[0,1]$, can be  concluded by standard methods; see e.g.\ Drees (2003), Theorem 6.2, for similar calculations for fixed $t$.

  In contrast, if the processes $X^{(i)}$ are discretely observed as discussed before, one may either define the quantile estimator analogously by replacing the marginal estimators with the interpolated counterparts, i.e.\ define
  $$ \hat x_{n,t}^* := \hat U_{n,t}^*(n/k_n) + \hat a_{n,t}^*(n/k_n)\frac{(np_n/k_n)^{-\hat\gamma_{n,t}^*}-1}{\hat\gamma_{n,t}^*}, \quad t\in[0,1],
  $$
  or one interpolates the quantile estimators between the observed points, that is, one considers
  $ \langle \hat x_{n,\cdot} \rangle_{n,t}$, $t\in[0,1]$. By lengthy, but simple calculations it can be concluded from Theorem \ref{theo:mainmarg} that, under the conditions given there, both estimators asymptotically  behave as the original estimator $ \hat x_{n,t}$, uniformly for $t\in[0,1]$.
\end{remark}

\subsection{Estimating the exponent measure}

For $u>0$ and Borel sets $E\subset C[0,1]$, let $\nu_u:=uP\{u^{-1}\xi^{(1)}\in E\}$. It is well known that $\lim_{u\to\infty} \nu_u(E)=\nu(E)<\infty$ for all Borel sets $E\subset C[0,1]$ such that $\inf_{z\in E}\|z\|_\infty>0$ and $\nu(\partial E)=0$ with $\nu$ defined in \eqref{eq:nudef}. (Here $\partial E$ denotes the topological boundary of $E$.) If the processes $\xi^{(i)}$ are observable, then one may estimate $\nu(E)$ by the following empirical counterpart of $\nu_{n/k_n}(E)$:
$$ \bar\nu_{n/k_n}(E) := \frac 1k \sum_{i=1}^n 1_{\textstyle \{k_n\xi^{(i)}/n\in E\}}.
$$
However, usually the marginal cdf's $F_t$ are unknown and must thus be replaced with suitable estimators in the definition of $\xi^{(i)}_t$ so that the resulting processes
$$ \hat\xi^{(i)}_t := \frac 1{1-\hat F_t(X_t^{(i)})}, \quad t\in[0,1], i\in\N,
$$
are continuous. For example, if $(\hat\gamma_{n,t})_{t\in[0,1]}$, $(\hat a_{n,t}(n/k_n))_{t\in[0,1]}$ and $(\hat U_{n,t}(n/k_n))_{t\in[0,1]}$ are consistent estimators of $(\gamma_t)_{t\in[0,1]}$, $(a_t(n/k_n))_{t\in[0,1]}$ and $(U_t(n/k_n))_{t\in[0,1]}$, respectively,  with continuous sample paths then one may consider
$$ \hat F_t(x) := \frac{n}{k_n} \bigg( 1+\hat\gamma_{n,t}\max\Big(\frac{x-\hat U_{n,t}(n/k_n)}{\hat a_{n,t}(n/k_n)}, -\frac 1{\hat\gamma_{n,t}^+}\Big)\bigg)^{1/\hat\gamma_{n,t}}.
$$
De Haan and Lin (2003) proved that the resulting estimator
\begin{equation}  \label{eq:hatnunkdef}
  \hat\nu_{n,k_n}(\cdot) := \frac 1k \sum_{i=1}^n 1_{\textstyle \{k_n\hat\xi^{(i)}/n\in\cdot\}}
\end{equation}
is consistent for $\nu$ if one uses the marginal estimators defined in \eqref{eq:gammadef}--\eqref{eq:Udef}. By consistency we mean that
$$ \hat\nu_{n,k_n}(E)  \,\stackrel{(P)}{\longrightarrow}\, \nu(E) $$
for all Borel sets $E\subset C[0,1]$ such that $\inf_{z\in E}\|z\|_\infty >0$ and $\nu(\partial E)=0$.
According to Daley and Vere-Jones (2008), Theorem 11.1.VII, and Daley and Vere-Jones (2003), Corollary A2.5.II, this is equivalent to
$$ d_c\big(\hat\nu_{n,k_n}|_{D_c},\nu|_{D_c}\big) \,\stackrel{(P)}{\longrightarrow}\, 0, \quad \forall\, c>0,
$$
where
$$ D_c:=\{z\in C[0,1]\mid \|z\|_\infty>c\} $$
and the distance between two measures $\mu,\tilde\mu$ on the Borel sets of $D_c$ is defined as
$$ d_c(\mu,\tilde \mu) := \inf\big\{\eps>0\mid \mu(F)\le\tilde\mu(F^\eps)+\eps, \tilde\mu(F)\le \mu(F^\eps)+\eps \text{ for all closed sets } F\subset D_c\big\}
$$
with
$$ F^\eps := \{ z\in C[0,1] \mid \|z-\tilde z\|\le \eps \text{ for some } \tilde z\in F\}.
$$

If the processes $X^{(i)}$ are only observed in the points $t_{n,j}$, $1\le j\le j_n$, then again one must apply some interpolation technique to estimate the exponent measure. As in Subsection 2.1, we discuss linear interpolation for general estimators of the exponent measure, but Theorem \ref{theo:expmeas} can easily be extended to more refined methods of smooth interpolation.

In what follows, we assume that a sequence of random measures $\hat\nu_n$ is given which is  consistent for $\nu$. We then define
$$ \hat\nu_n^* (E) := \hat\nu_n\{ z\in C[0,1] \mid \inter{z}_n\in E\}
$$
with $\inter{z}_n$ given in \eqref{eq:interdef}. For example, for $\hat\nu_{n,k_n}$ as in \eqref{eq:hatnunkdef} we obtain
$$ \hat\nu_{n,k_n}^*(\cdot) = \frac 1k \sum_{i=1}^n 1_{\textstyle
\{k_n\inter{\hat\xi^{(i)}}_n/n\in\cdot\}}.
$$
If the marginal estimators $\hat\gamma_{n,t}, \hat a_{n,t}(n/k_n)$ and $\hat U_{n,t}(n/k_n)$ only depend on $X_t^{(i)}$, $1\le i\le n$, then this estimator $\hat\nu_{n,k_n}^*$depends on the discrete observations only.

Without any further assumptions, consistency carries over from $\hat\nu_n$ to $\hat\nu_n^*$.
\begin{theorem} \label{theo:expmeas}
  If $\hat\nu_n(E)\stackrel{(P)}{\longrightarrow}\nu(E)$ for all Borel sets $E\subset C[0,1]$ such that $\inf\{\|z\|_\infty| z\in E\}>0$ and $\nu(\partial E)=0$, then this convergence also holds for $\hat\nu_n^*$.
\end{theorem}

To the best of our knowledge, no result on the asymptotic normality of an estimator of the exponent measure is known. Indeed, since here estimators are random measures, for such a result one has to consider a family $\mathcal{G}\subset C[0,1]$ of test functions and prove that $\big(\lambda_n^{-1}(\int g \, d\hat\nu_n- \int g\, d\nu)\big)_{g\in\mathcal{G}}$ converges to a Gaussian process uniformly on $\mathcal{G}$. However, no family $\mathcal{G}$ suggests itself, and it seems likely that the choice of a suitable family depends on the applications one has in mind. We thus refrain from investigating the asymptotic normality of $\hat\nu_n^*$.

\section{Proofs}

\begin{proofof} Theorem \ref{theo:simple}. \rm
  We only verify \eqref{eq:dUconv} as the other assertions can be proved by similar arguments.

  For $t\in[0,t_{n,1}]$, one has
    \begin{eqnarray*}
    \lambda_n^{-1} \frac{\inter{\hat U_n}_{n,t}(n/k_n)- U_t(n/k_n)}{a_t(n/k_n)} -B_t
    & = & \frac{\hat U_{n,t_{n,1}}(n/k_n)-U_{t_{n,1}}(n/k_n)}{a_{t_{n,1}}(n/k_n)}\cdot \lambda_n^{-1}\Big(\frac{a_{t_{n,1}}(n/k_n)}{a_t(n/k_n)}-1\Big) \\
    & & { } + \lambda_n^{-1}\frac{\hat U_{n,t_{n,1}}(n/k_n)-U_{t_{n,1}}(n/k_n)}{a_{t_{n,1}}(n/k_n)}-B_{t_{n,1}}\\
    & & { }+ \lambda_n^{-1} \frac{U_{t_{n,1}}(n/k_n)-U_t(n/k_n)}{a_t(n/k_n)}\\
    & & { }+ B_{t_{n,1}}-B_t.
    \end{eqnarray*}
 Condition \eqref{eq:Uconv} shows that the second term on the right-hand side tends to 0 in probability uniformly for all $t\in[0,t_{n,1}]$. In  particular, the first factor of the first term is stochastically bounded. Hence the first term tends to 0 by condition \eqref{eq:asmooth}. The last two summands vanish uniformly  by \eqref{eq:Usmooth} and the pathwise continuity of  $B$.
 Likewise, one can prove
 $$ \sup_{t\in[t_{n,j_n},1]} \bigg|\lambda_n^{-1} \frac{\inter{\hat U_n}_{n,t}(n/k_n)- U_t(n/k_n)}{a_t(n/k_n)} -B_t\bigg| \,\stackrel{(P)}{\longrightarrow}\, 0.
 $$

 Similarly, for $2\le j\le j_n$ and $i\in\{j-1,j\}$, one has uniformly for all $t\in (t_{n,j-1}-t_{n,j}]$
  \begin{eqnarray*}
    \lambda_n^{-1} \frac{\hat U_{n,t_{n,i}}(n/k_n)- U_t(n/k_n)}{a_t(n/k_n)} -B_t
    & = &  \frac{\hat U_{n,t_{n,i}}(n/k_n)- U_{t_{n,i}}(n/k_n)}{a_{t_{n,i}}(n/k_n)} \cdot
    \lambda_n^{-1} \Big(\frac{a_{t_{n,i}}(n/k_n)}{a_t(n/k_n)} -1\Big) \\
    & &  { } +\lambda_n^{-1} \frac{\hat U_{n,t_{n,i}}(n/k_n)- U_{t_{n,i}}(n/k_n)}{a_{t_{n,i}}(n/k_n)} -B_{t_{n,i}}\\
    & & { }+ \lambda_n^{-1} \frac{U_{t_{n,i}}(n/k_n)-U_t(n/k_n)}{a_t(n/k_n)}\\
    & & { }  + B_{t_{n,i}}-B_t\\
    & = & o_P(1)
  \end{eqnarray*}
  by \eqref{eq:Uconv}, \eqref{eq:asmooth}, \eqref{eq:Usmooth} and the continuity of $B$. Hence, with $c_{n,t}:= (t_{n,j}-t)/(t_{n,j}-t_{n,j-1})\in[0,1]$ one may conclude  that
  \begin{eqnarray*}
    \lefteqn{\lambda_n^{-1} \frac{\inter{\hat U_n}_{n,t}(n/k_n)- U_t(n/k_n)}{a_t(n/k_n)} -B_t}\\
    & = & c_{n,t} \Big( \lambda_n^{-1} \frac{\hat U_{n,t_{n,j-1}}(n/k_n)- U_t(n/k_n)}{a_t(n/k_n)} -B_t\Big)+ (1-c_{n,t})\Big( \lambda_n^{-1} \frac{\hat U_{n,t_{n,j}}(n/k_n)- U_t(n/k_n)}{a_t(n/k_n)} -B_t\Big)\\
    & = & o_P(1)
  \end{eqnarray*}
  uniformly for all $t\in(t_{n,j-1},t_{n,j}]$ and $2\le j\le j_n$, which proves assertion \eqref{eq:dUconv}.
\end{proofof}

The next lemma states some consequences of the conditions (M($\lambda_n$)) and (S($\lambda_n$)) that will be useful for the proof of Theorem \ref{theo:mainmarg}.
\begin{lemma} \label{lemma:ramification}
  If the conditions  (M($\lambda_n$)) and (S($\lambda_n$)) hold, then  for all $\tilde\tau>\tau$ there exists $n_{\tilde\tau}$ such that for all $n>n_{\tilde\tau}$
  \begin{equation}  \label{eq:lowbd}
   U_s(n/k_n)+\tilde\tau a_s(n/k_n)\ge U_t(n/k_n)+\tau a_t(n/k_n)\quad \forall s,t\in[0,1], |s-t|\le\delta_n.
  \end{equation}
  Moreover,
  \begin{equation} \label{eq:abound}
   \sup_{|s-t|\le\delta_n} \frac{a_t(n/k_n)}{a_s(n/k_n)}= O(1).
  \end{equation}
  In particular, for all  $\tilde\tau>\tau$ and $\eps>0$
  \begin{equation}  \label{eq:altS}
     \sup_{|s-t|\le \delta_n} P\Big\{ \frac{|X_s-X_t|}{a_t(n/k_n)}>\eps\lambda_n, X_s>U_t(n/k_n)+\tilde\tau a_t(n/k_n)\Big\} = o(\lambda_n k_n/n).
  \end{equation}
\end{lemma}
\begin{proof}
   Suppose assertion \eqref{eq:lowbd} were wrong. Then there exist sequences $s_n,t_n\in[0,1]$, $n\in\N$, such that $|s_n-t_n|\le\delta_n$ for all $n\in\N$ and
   $$ U_{s_n}(n/k_n)+\tilde\tau a_{s_n}(n/k_n)< U_{t_n}(n/k_n)+\tau a_{t_n}(n/k_n).
   $$
    Because $[0,1]$ is compact, we may assume w.l.o.g.\ that both sequences $(s_n)_{n\in\N}$ and $(t_n)_{n\in\N}$ converge to some limit $t\in[0,1]$. For any $\tau'\in(\tau,\tilde\tau)$ and $\zeta>0$, let $y_n':=(1+\gamma_{s_n}\tau')^{1/\gamma_{s_n}}$ and $y_n:=\big(1+\gamma_{t_n}(\tau+2\zeta\lambda_n)\big)^{1/\gamma_{t_n}}$. In view of condition (M($\lambda_n$)),  one has eventually
    $$ U_{s_n}(y_n'n/k_n) < U_{s_n}(n/k_n)+ a_{s_n}(n/k_n)(\tau'+\zeta\lambda_n) < U_{t_n}(n/k_n)+\tau a_{t_n}(n/k_n)
    $$
    and
    $$  U_{t_n}(y_nn/k_n) > U_{t_n}(n/k_n)+a_{t_n}(n/k_n)(\tau+\zeta\lambda_n). $$
    Note that by the definition of $U_{t_n}$ one has $P\{X_{t_n}>x\}> k_n/(y_n n)$ for all $x<U_{t_n}(y_n n/k_n)$.
    Thus, using condition (S($\lambda_n$)), we may conclude
    \begin{eqnarray}
      \frac 1{y_n} & < & \frac n{k_n} P\Big\{X_{t_n}>U_{t_n}(n/k_n)+a_{t_n}(n/k_n)(\tau+\zeta\lambda_n)\Big\} \nonumber\\
      & \le  & \frac n{k_n} P\Big\{X_{s_n}>U_{t_n}(n/k_n)+a_{t_n}(n/k_n)\tau\Big\}+o(\lambda_n)\nonumber\\
      & \le  & \frac n{k_n} P\Big\{X_{s_n}>U_{s_n}(ny_n'/k_n)\Big\}+o(\lambda_n)\nonumber\\
      & \le & \frac 1{y_n'} + o(\lambda_n). \label{eq:ycomp}
    \end{eqnarray}
    On the other hand, the continuity of the function $(\gamma_t)_{t\in[0,1]}$ implies $y_n'-y_n\to (1+\gamma_t\tau')^{1/\gamma_t}-(1+\gamma_t\tau)^{1/\gamma_t}>0$, in contradiction to \eqref{eq:ycomp}.
    Hence assertion \eqref{eq:lowbd} is proved.

    Using this inequality and interchanging the roles of $s$ and $t$ in condition (S($\lambda_n$)) yields
    \begin{equation}  \label{eq:Svar1}
       \sup_{|s-t|\le \delta_n} P\Big\{ \frac{|X_s-X_t|}{a_s(n/k_n)}>\eps\lambda_n, X_s>U_t(n/k_n)+\tilde\tau a_t(n/k_n)\Big\} = o(\lambda_n k_n/n)
    \end{equation}
    for all $\tilde\tau>\tau$ and $\eps>0$. Now suppose assertion \eqref{eq:abound} were wrong, i.e.\ there exist $s_n,t_n\in[0,1]$ such that $|s_n-t_n|\le\delta_n$ and $a_{t_n}(n/k_n)/a_{s_n}(n/k_n)\to \infty$. Obviously, condition (S($\lambda_n$)) for a specific $\tau$ implies (S($\lambda_n$)) for all $\tau'\in(\max(\tau,0), \tau_{\max})$. Choose some $\tau''\in(\max(\tau,0),\tau')$ and $\tau'''\in(\tau',\tau_{\max})$. Then, by condition (M($\lambda_n$)) and \eqref{eq:lowbd} (applied with $(\tau''',\tau')$ instead of $(\tilde\tau,\tau)$), one has eventually
    \begin{eqnarray*}
      U_{s_n}(n/(2k_n)) & > & U_{s_n}(n/k_n)+a_{s_n}(n/k_n)\Big(\frac{2^{-\gamma_{s_n}}-1}{\gamma_{s_n}}-\lambda_n\Big)\\
      & \ge & U_{t_n}(n/k_n)+\tau'a_{t_n}(n/k_n) + a_{s_n}(n/k_n)\Big(\frac{2^{-\gamma_{s_n}}-1}{\gamma_{s_n}}-\lambda_n-\tau'''\Big)\\
      & \ge & U_{t_n}(n/k_n)+\tau''a_{t_n}(n/k_n).
    \end{eqnarray*}
    Hence, \eqref{eq:Svar1} implies that for sufficiently large $n$
    \begin{eqnarray*}
      2 & \le & \frac n{k_n}P\big\{X_{s_n}>  U_{t_n}(n/k_n)+\tau''a_{t_n}(n/k_n)\big\} \\
      & \le & \frac n{k_n}P\big\{X_{t_n}>  U_{t_n}(n/k_n)+\tau''a_{t_n}(n/k_n)-\lambda_n a_{s_n}(n/k_n)\big\}+o(\lambda_n)\\
      & \le & \frac n{k_n}P\big\{X_{t_n}>  U_{t_n}(n/k_n)\big\}+o(\lambda_n)\\
      & \le & 1+o(\lambda_n).
    \end{eqnarray*}
    As this is obviously a contradiction, assertion  \eqref{eq:abound} is proved. Now \eqref{eq:altS} follows readily from \eqref{eq:Svar1}.
\end{proof}

\begin{proofof} Theorem \ref{theo:mainmarg}. \rm
  We first establish \eqref{eq:Uapprox} in the case $\underline y>y_{\tau'}:=\sup_{t\in[0,1]}(1+\gamma_t\tau')^{\gamma_t}$ for some fixed $\tau'\in(\tau,\tau_{\max})$. Suppose this assertion were wrong. Then there exist sequences $s_n,t_n\in[0,1]$, $y_n\in[\underline y,\bar y]$ and some $\eps>0$ such that $|s_n-t_n|\le \delta_n$ and
  \begin{equation} \label{eq:Uneg}
    \frac{U_{s_n}(y_nn/k_n)-U_{t_n}(n/k_n)}{a_{t_n}(n/k_n)} - \frac{y_n^{\gamma_{t_n}}-1}{\gamma_{t_n}} \not\in[-\eps\lambda_n,\eps\lambda_n], \quad\forall\, n\in\N.
  \end{equation}
  We may also assume that $t_n\to t\in[0,1]$ and $y_n\to y\in[\underline y,\bar y]$, and that the left hand side of \eqref{eq:Uneg} always exceeds $\eps\lambda_n$ or that it is always less than $-\eps\lambda_n$, as this holds for a suitable subsequence. We will only consider the former case, because the latter can be treated analogously.

  By the choice of $\underline y$, the expression $(y_n^{\gamma_{t_n}}-1)/\gamma_{t_n}+\eps\lambda_n$ exceeds $\tau'$ for sufficiently large $n$. Hence
  $$ U_{s_n}(y_nn/k_n) > U_{t_n}(n/k_n) + a_{t_n}(n/k_n)\Big( \frac{y_n^{\gamma_{t_n}}-1}{\gamma_{t_n}}+\eps\lambda_n\Big)
  $$
  implies
  \begin{eqnarray*}
  \frac 1{y_n} & <  & \frac n{k_n} P\Big\{ X_{s_n}> U_{t_n}(n/k_n)+a_{t_n}(n/k_n)\Big( \frac{y_n^{\gamma_{t_n}}-1}{\gamma_{t_n}}+\eps\lambda_n\Big)\Big\}\\
  & \le & \frac n{k_n} P\Big\{ X_{t_n}> U_{t_n}(n/k_n)+a_{t_n}(n/k_n)\Big( \frac{y_n^{\gamma_{t_n}}-1}{\gamma_{t_n}}+\frac{\eps}{2} \lambda_n\Big)\Big\}+o(\lambda_n),
  \end{eqnarray*}
  where in the last step we have applied \eqref{eq:altS}. Let $\tilde y_n:= \big(y_n^{\gamma_{t_n}}+\eps\lambda_n\gamma_{t_n}/4\big)^{1/\gamma_{t_n}}$. In view of condition (M($\lambda_n$)), one has for sufficiently large $n$
  $$
    U_{t_n}(\tilde y_nn/k_n) < U_{t_n}(n/k_n)+a_{t_n}(n/k_n)\Big(\frac{\tilde y_n^{\gamma_{t_n}}-1}{\gamma_{t_n}}+\frac\eps 4\lambda_n\Big) =  U_{t_n}(n/k_n)+a_{t_n}(n/k_n)\Big(\frac{y_n^{\gamma_{t_n}}-1}{\gamma_{t_n}}+\frac\eps 2\lambda_n\Big).
  $$
  Therefore,
  \begin{eqnarray*}
  \frac 1{y_n} & <  & \frac n{k_n}P\big\{ X_{t_n}> U_{t_n}(n\tilde y_n/k_n) \big\} +o(\lambda_n)\\
   & \le & \frac 1{\tilde y_n} + o(\lambda_n)\\
   & = & \frac 1{y_n}\Big(1+\frac\eps 4 \lambda_n\gamma_{t_n}y_n^{-\gamma_{t_n}}\Big)^{-1/\gamma_{t_n}}+o(\lambda_n),
  \end{eqnarray*}
  which implies
  $$ 1- \Big(1+\frac\eps 4 \lambda_n\gamma_{t_n}y_n^{-\gamma_{t_n}}\Big)^{-1/\gamma_{t_n}} = o(\lambda_n). $$
  This, however, contradicts the fact that
  $$ \Big(1+\frac\eps 4 \lambda_n\gamma_{t_n}y_n^{-\gamma_{t_n}}\Big)^{-1/\gamma_{t_n}}= 1-\frac\eps 4 \lambda_n y_n^{-\gamma_{t_n}} + O(\lambda_n^2) = 1-\frac\eps 4 \lambda_n y^{-\gamma_t} + O(\lambda_n^2).
  $$

  Next we prove  \eqref{eq:Uapprox} for arbitrary $\underline y>0$. Let $c:=\underline y/(2y_{\tau'})$ so that $y/c\in[2y_{\tau'},2y_{\tau'}\bar y/\underline y]$ for $y\in[\underline y,\bar y]$. Furthermore, define
  $$ m_n := \inf\big\{ l\in\N\mid cn/k_n=ch(n)\le h(l)=l/k_l\big\}, $$
  so that $h(m_n-1)<ch(n)\le h(m_n)$.
  The regular variation of the function $h$ implies $m_n\sim c^{1/\kappa}n$. Moreover, by our assumptions on $k_n$ and $\lambda_n$,
  \begin{equation}  \label{eq:mnasymp}
    \frac{h(m_n-1)}{h(m_n)} = \frac{m_n-1}{m_n}\cdot \frac{k_{m_n}}{k_{m_n-1}} = (1-m_n^{-1})(1+o(\lambda_{m_n})) = 1+o(\lambda_n).
  \end{equation}
  An application of  \eqref{eq:Uapprox} in the special case considered above and of (M($\lambda_n$)) shows that
  \begin{eqnarray}
    \lefteqn{\frac{U_s(yn/k_n)-U_t(n/k_n)}{a_t(n/k_n)}}\nonumber\\
     & = & \frac{U_s(ch(n)y/c)-U_t(h(n))}{a_t(h(n))} \nonumber\\
      & \le & \frac{U_s(h(m_n)y/c)-U_t(h(m_n))}{a_t(h(m_n))} \cdot \frac{a_t(h(m_n))}{a_t(h(n))}+ \frac{U_t(h(m_n))-U_t(h(n))}{a_t(h(n))} \nonumber  \\
      & \le & \Big(\frac{(y/c)^{\gamma_t}-1}{\gamma_t}+o(\lambda_{m_n})\Big)\cdot \frac{a_t(h(m_n))}{a_t(h(n))}+ \frac{(h(m_n)/h(n))^{\gamma_t}-1}{\gamma_t} + o(\lambda_n)  \label{eq:Uapprox2}
  \end{eqnarray}
  uniformly for $y\in[\underline y,\bar y]$ and $s,t\in[0,1]$ such that $|s-t|\le\delta_n$. Note that by (M($\lambda_n$))
  $$  \frac{U_t(h(m_n)y)-U_t(h(m_n))}{a_t(h(m_n))} - \frac{y^{\gamma_t}-1}{\gamma_t}=o(\lambda_{m_n}) $$
  and
  $$ \frac{U_t(h(m_n)y)-U_t(h(n))}{a_t(h(n))} - \frac{\big(h(m_n)/h(n)y\big)^{\gamma_t}-1}{\gamma_t} = o(\lambda_n)
  $$
  uniformly for $y\in[\underline y,\bar y]$ and $t\in[0,1]$. Thus
  \begin{eqnarray*}
    \frac{a_t(h(m_n))}{a_t(h(n))}\Big(\frac{y^{\gamma_t}-1}{\gamma_t}+o(\lambda_{m_n})\Big)
    & = & \frac{a_t(h(m_n))}{a_t(h(n))}\cdot \frac{U_t(h(m_n)y)-U_t(h(m_n))}{a_t(h(m_n))}\\
    & = &
    \frac{U_t(h(m_n)y)-U_t(h(n))}{a_t(h(n))} - \frac{U_t(h(m_n))-U_t(h(n))}{a_t(h(n))} \\
    & = & \Big(\frac{h(m_n)}{h(n)}\Big)^{\gamma_t} \frac{y^{\gamma_t}-1}{\gamma_t} +o(\lambda_n).
  \end{eqnarray*}
  Since $h(m_n)/h(n)=c+o(\lambda_n)$ by \eqref{eq:mnasymp} and the definition of $m_n$, and $\lambda_{m_n}=O(\lambda_n)$ by assumption, we may conclude
  $$ \frac{a_t(h(m_n))}{a_t(h(n))} = c^{\gamma_t}+o(\lambda_n) $$
  uniformly for  $t\in[0,1]$.
  Therefore, the right hand side of \eqref{eq:Uapprox2} equals
  $(y^{\gamma_t}-1)/\gamma_t+o(\lambda_n)$. Likewise, one can show that
   \begin{eqnarray*}
    \lefteqn{\frac{U_s(yn/k_n)-U_t(n/k_n)}{a_t(n/k_n)}}\\
      & \ge & \Big(\frac{(y/c)^{\gamma_t}-1}{\gamma_t}+o(\lambda_{m_n-1})\Big)\cdot \frac{a_t(h(m_n-1))}{a_t(h(n))}+ \frac{(h(m_n-1)/h(n))^{\gamma_t}-1}{\gamma_t} + o(\lambda_n)\\
      & \ge & \frac{\big(yh(m_n-1)/(ch(n))\big)^{\gamma_t}-1}{\gamma_t}+o(\lambda_n) \\
      & =  & \frac{y^{\gamma_t}-1}{\gamma_t}+o(\lambda_n).
    \end{eqnarray*}
   Combining these bounds, we obtain \eqref{eq:Uapprox} in the general case.

  Equation \eqref{eq:Usmooth} is an obvious consequence for $y=1$.

  Combining condition (M($\lambda_n$)) with \eqref{eq:Usmooth} yields
  $$ \frac{U_s(yn/k_n)-U_t(n/k_n)}{a_s(n/k_n)} = \frac{U_s(n/k_n)-U_t(n/k_n)}{a_s(n/k_n)} +\frac{y^{\gamma_s}-1}{\gamma_s} + o(\lambda_n) = \frac{y^{\gamma_s}-1}{\gamma_s} + o(\lambda_n).
  $$
  On the other hand, by \eqref{eq:Uapprox}
  $$ \frac{U_s(yn/k_n)-U_t(n/k_n)}{a_t(n/k_n)} = \frac{y^{\gamma_t}-1}{\gamma_t} + o(\lambda_n),
  $$
  so that
  \begin{equation} \label{eq:afrac}
   \frac{a_t(n/k_n)}{a_s(n/k_n)} = \frac{(y^{\gamma_s}-1)/\gamma_s+o(\lambda_n)}{(y^{\gamma_t}-1)/\gamma_t+o(\lambda_n)}
   = \frac{(y^{\gamma_s}-1)/\gamma_s}{(y^{\gamma_t}-1)/\gamma_t}(1+o(\lambda_n))
  \end{equation}
  for all $y>1$ uniformly for $|s-t|\le\delta_n$. In particular
  $$ \frac{2^{\gamma_t}+1}{2^{\gamma_s}+1} = \frac{4^{\gamma_t}-1}{4^{\gamma_s}-1}\cdot \frac{2^{\gamma_s}-1}{2^{\gamma_t}-1} = 1+o(\lambda_n),
  $$
  which implies
  $$ \frac{2^{\gamma_s}}{2^{\gamma_s}+1}\big(2^{\gamma_t-\gamma_s}-1\big) =\frac{2^{\gamma_t}+1}{2^{\gamma_s}+1}-1 = o(\lambda_n)
  $$
   uniformly for $|s-t|\le\delta_n$, and hence \eqref{eq:gammasmooth}.

  Finally, it follows that the right hand side of  \eqref{eq:afrac} equals $1+o(\lambda_n)$ uniformly for $|s-t|\le\delta_n$, because $\gamma\mapsto (y^\gamma-1)/\gamma$ is differentiable, which proves \eqref{eq:asmooth}.
\end{proofof}

\begin{proofof} Theorem \ref{theo:expmeas}. \rm
  Denote the modulus of continuity of a function $z\in C[0,1]$ by
  $$ \omega_z(\delta) := \sup\{ |z(x)-z(y)|\mid x,y\in[0,1], |x-y|\le\delta\}.
  $$
   Since $\nu(D_c)<\infty$ and the closed sets $E_c^{(\delta,\zeta)}:=\{z\in D_c \mid\omega_z(\delta)\ge\zeta\}$ converge to the empty set as $\delta\downarrow 0$ for all $c,\zeta>0$, to each $\zeta,\iota>0$ there exists $\delta=\delta(\zeta,\iota)>0$ such that $\nu(E_c^{(\delta,\zeta)})<\iota$ holds. Moreover, $\omega_z(\delta)\ge 3\zeta$ and $\|z-\tilde z\|_\infty\le\zeta$ imply $\omega_{\tilde z}(\delta)\ge\zeta$.
  Therefore, on the event $\{d_c(\hat\nu_n|_{D_c},\nu|_{D_c})< \zeta\}$, one has
  $$\hat\nu_n(E_c^{(\delta,3\zeta)})\le \nu\big((E_c^{(\delta,3\zeta)})^\zeta\big)+\zeta\le \nu(E_c^{(\delta,\zeta)})+\zeta<\iota+\zeta.
  $$

  Next, fix some $\eps\in (0,c)$ and let $\zeta:=\eps/12$ and $\iota=\eps/4$. Because $\|z-\inter{z}_n\|_\infty\le 2\omega_z(\delta_n)$, from $\inter{z}_n\in F$ and $\omega_z(\delta_n)<\eps/4$ one may conclude $z\in F^{\eps/2}$.  Hence, on the event $\{d_{c-\eps}(\hat\nu_n|_{D_{c-\eps}},\nu|_{D_{c-\eps}})\le \eps/2\}$, one has for sufficiently large $n$ (such that $\delta_n\le \delta(\eps/12,\eps/4)$) and all closed sets $F\subset D_c$
  \begin{eqnarray}
   \hat\nu_n^*(F) & \le & \hat\nu_n\big\{ z\in C[0,1] \mid \inter{z}_n\in F, \omega_z(\delta_n)<\eps/4\big\} + \hat\nu_n\big\{ z\in D_c \mid\omega_z(\delta_n)\ge \eps/4\big\} \nonumber\\
   & \le & \hat\nu_n(F^{\eps/2})+\hat\nu_n(E_c^{(\delta,3\zeta)}) \nonumber\\
   & \le & \nu(F^\eps)+\eps/2+\iota+\zeta  \nonumber\\
   & \le & \nu(F^\eps)+\eps.  \label{eq:nubound1}
  \end{eqnarray}
  Likewise, on $\{d_{c-\eps}(\hat\nu_n|_{D_{c-\eps}},\nu|_{D_{c-\eps}})<\eps/2\}$
  \begin{eqnarray}
   \nu(F) & \le & \hat\nu_n(F^{\eps/2})+\eps/2\nonumber\\
    & \le &  \hat\nu_n\big\{ z\in C[0,1] \mid \inter{z}_n\in F^{\eps}, \omega_z(\delta_n)<\eps/4\big\}+\hat\nu_n\{z\in D_{c-\eps}\mid \omega_z(\delta_n)\ge\eps/4\}+\eps/2 \nonumber\\
    & \le & \hat\nu_n^*(F^\eps) +\eps.  \label{eq:nubound2}
   \end{eqnarray}
   A combination of \eqref{eq:nubound1} and  \eqref{eq:nubound2} shows that $\big\{d_{c-\eps}(\hat\nu_n|_{D_{c-\eps}},\nu|_{D_{c-\eps}})\le\eps/2\big\}\subset \big\{d_c(\hat\nu_n^*|_{D_c},\nu|_{D_c})\le\eps\big\}$ for all $c>\eps>0$. Hence, the consistency of $\hat\nu_n$ implies that of $\hat\nu_n^*$.
\end{proofof}

\bigskip

{\bf Acknowledgement}: L.\ de Haan and F.\ Turkman have been partly funded by FCT - Funda\c{c}\~{a}o para a Ci\^{e}ncia e a Tecnologia, Portugal,
through the project UID/MAT/00006/2013. H.\ Drees has been partly supported by DFG project DR 271/6-2 within the research unit FOR 1735. We thank two anonymous referees whose constructive remarks led to an improvement of the presentation.

\bigskip

\noindent {\large\bf References}
   \smallskip

\parskip1.2ex plus0.2ex minus0.2ex

\rueck
 Adler, R.J. (1990). {\em An Introduction to Continuity, Extrema, and Related Topics for General Gaussian Processes.} IMS Lecture Notes {\bf 12}.

\rueck
 Albin, J.M.P. (1990). On extremal theory for stationary processes. {\em Ann.\ Probab.} {\bf 18}, 92--128.

\rueck
  Buhl, S., and Kl\"{u}ppelberg, C. (2016). Anisotropic Brown-Resnick space-time processes: estimation and model assessment. {\em Extremes} {\bf 19}, 627--660.

\rueck
  Daley, D.J., and Vere-Jones , D. (2003). {\em An Introduction to the Theory of Point Processes, Vol.\ I: Elementary Theory and Methods.} Springer.

\rueck
  Daley, D.J., and Vere-Jones , D. (2008). {\em An Introduction to the Theory of Point Processes, Vol.\ II: General Theory and Structure.} Springer.

\rueck
    Dombry, C., \'{E}yi-Minko, F., and Ribatet, M. (2013). Conditional simulation of max-stable processes. {\em  Biometrika } {\bf 100}, 111--124.

\rueck
  Drees, H. (2003). Extreme Quantile estimation for dependent data with applications to finance.  {\em Bernoulli} {\bf 9}, 617--657.

\rueck
  Einmahl, J.H.J., and Lin,T. (2006). Asymptotic normality of extreme value estimators on $C[0,1]$. {\it Ann.\ Statist.} {\bf 34}, 469--492.

\rueck
  Falk, M., Hofmann, M., and Zott, M. (2015). On generalized max-linear models and their statistical interpolation. {\em J.\ Appl.\ Probab.} {\bf 52}, 736--751.

\rueck
  Fuentes, M., Henry, J., and Reich, B. (2013). Nonparametric spatial models for extremes: application to extreme temperature data. {\em Extremes} {\bf 16}, 75--101.

\rueck
  Genton, M.G., Padoan, S.A., and Dang, H. (2015). Multivariate max-stable spatial processes. {\em Biometrika} {\bf 102}, 215--230.

\rueck
  de Haan, L., and Ferreira, A. (2006). {\em Extreme Value Theory.}
  Springer.

\rueck
  de Haan, L., and  Pereira, T.T. (2006). Spatial extremes: models for the stationary case.   {\em Ann. Statist.} {\bf 34}, 146--168.

\rueck
  de Haan, L., and Lin, T. (2003). Weak consistency of extreme value estimators in $C[0,1]$. {\it Ann.\ Statist.} {\bf 31}, 1996--2012.

\rueck
 Lehmann, E.A., Phatak, A., Stephenson, A.G., and Lau, R. (2016). Spatial modelling framework for the characterisation of rainfall extremes at different durations and under climate change. {\em Environmetrics} {\bf 27}, 239--251.

\rueck
  Oesting, M., and Schlather, M. (2014). Conditional sampling for max-stable processes
with a mixed moving maxima representation. {\em Extremes} {\bf 17}, 157--192.

\rueck
  Oesting, M., Schlather, M., and Friederichs, P. (2017). Statistical post-processing of forecasts for extremes
using bivariate brown-resnick processes
with an application to wind gusts. {\em Extremes} {\bf 20}, 309--332.

\rueck
  Piterbarg, V.I.  (2004). Discrete and continuous time extremes of Gaussian processes. {\em Extremes} {\bf 7}, 161-–177.

\rueck
  Turkman, K.F. (2012).  Discrete and continuous time extremes of stationary processes.  In: {\em Handbook of Statistics} {\bf 30}, T.\ Subba Rao, S.\ Subba Rao, and C.R.\ Rao (eds.),  565--581, Elsevier.

\rueck
  Wang, Y., and Stoev, S.A. (2011). Conditional sampling for spectrally discrete max-stable random fields. {\em Adv. Appl. Probab.} {\bf 43}, 461-–483.

\end{document}